\pdfoutput=1
\documentclass[preprint,ECP,NODS,nobera]{ejpecp}

\SHORTTITLE{The Pile Process on a Cycle}
\TITLE{The Pile Process on a Cycle}

\AUTHORS{Itai~Benjamini \and Eric~Shellef}

\SUBMITTED{June 3, 2026}
\VOLUME{0}
\YEAR{2026}
\PAPERNUM{0}

\ABSTRACT{Consider a finite particle system in which, at each step, one particle from a vertex with a lower neighboring vertex moves to a neighboring vertex of minimum height. For the cycle \(C_n\), started with \(n\) particles at one vertex and no particles elsewhere, we prove that the expected stabilization time is at most a constant multiple of \(n^3\). The deterministic transportation lower bound is of order \(n^2\). Simulations suggest that the true order should be near \(n^3\). Chat GPT was used to reincarnate notes from 2007.}

\makeatletter
\renewcommand{\PDFFIELDS}{%
\if@ejpecp@preprint
\else
\hypersetup{%
pdftitle={\@JOURNALA\ \@VOLUME\ (\@YEAR), \papernum@name\ \@PAPERNUM, \@doiprefix\@DOI}}%
\hypersetup{pdfproducer={\@JOURNAL\ Managing Editor https://\@URL}}%
\hypersetup{pdfauthor={Please see \@doiprefix\@DOI}}%
\fi
\hypersetup{pdfkeywords={}}%
\hypersetup{pdfcreator={LaTeX with ejpecp.cls \@nameuse{ver@ejpecp.cls}}}%
\hypersetup{pdfsubject={}}%
}
\renewcommand{\FIRSTPAGE}{%
\setcounter{page}{\@PAGESTART}%
\title{\small
\let\orig@thanks\thanks
\ifx\@EJP\undefined\else
\if@ejpecp@preprint
\phantom{\@EJPLOGO}\par
\else
\@EJPLOGO
\fi
\bigskip
\if@ejpecp@preprint
\null\null
\else
\centerline{\@JOURNALA\ \textbf{\@VOLUME} (\@YEAR),\
\papernum@name\ \@PAPERNUM, \,\@PAGESTART--\@PAGEEND.}%
\centerline{ISSN:\ \texttt{\href{\@FULLURL}{\@ISSN}} %
\ifx\@DOI\undefined\else
\ \printdoi
\fi}%
\fi
\bigskip
\bigskip
\bgroup
\@ifundefined{no@title@thanks@support}{%
\let\thanks\title@thanks
\let\@makefnmark\title@makefnmark
\let\support\thanks
}{}%
\Large\bfseries\@TITLE\par
\egroup
\fi
\ifx\@ECP\undefined
\else
\parbox[t]{9cm}{%
\if@ejpecp@preprint
\else
\@JOURNALA\ \textbf{\@VOLUME} (\@YEAR),
\papernum@name\ \@PAPERNUM, \@PAGESTART--\@PAGEEND.\\
\ifx\@DOI\undefined
\else
\ \printdoi
\fi \\
ISSN:\ \texttt{\href{\@FULLURL}{\@ISSN}}%
\fi
}%
\hfill
\if@ejpecp@preprint
\phantom{\@ECPLOGO}%
\else
\@ECPLOGO
\fi\\
\bigskip
\bigskip
\bgroup
\@ifundefined{no@title@thanks@support}{%
\let\thanks\title@thanks
\let\@makefnmark\title@makefnmark
\let\support\thanks
}{}%
\Large\@TITLE\par
\egroup
\fi
}
\date{%
\ifx\@DEDICATORY\undefined
\else
\noindent
\emph{\small\sffamily\@DEDICATORY}%
\fi}%
\maketitle\thispagestyle{empty}%
\begin{abstract}%
\noindent
\@ABSTRACT\par\vskip 1em\relax
{\footnotesize
\if@ejpecp@preprint
\mbox{}\par
\else
\noindent
Submitted to \@JOURNAL\ on \@SUBMITTED,
final version accepted on \@ACCEPTED.\par
\fi
\ifx\@ARXIVID\undefined
\else
\noindent
Supersedes
\texttt{\href{https://arXiv.org/abs/\@ARXIVID}{arXiv:\@ARXIVID}}.%
\fi\par
\ifx\@HALID\undefined
\else
\noindent
Supersedes
\texttt{\href{https://hal.archives-ouvertes.fr/\@HALID}{HAL:\@HALID}}.%
\fi\par
}
\end{abstract}
\smallskip
}
\makeatother

\newcommand{\E}{\mathbb E}
\newcommand{\N}{\mathbb N}

\newcommand{\diam}{\operatorname{diam}}

\begin{document}

\section{The process}\label{sec:process}

Let \(G=(V,E)\) be a finite connected graph. A state of the pile process is a function
\[
\eta:V\to \N\cup\{0\},
\]
where \(\eta(x)\) is the number of particles, or the \emph{height}, at \(x\). A vertex \(x\) is \emph{free} in the state \(\eta\) if it has a neighbor of smaller height. Write
\[
F(\eta)=\{x\in V: \exists y\sim x \text{ with } \eta(y)<\eta(x)\}.
\]
For \(x\in F(\eta)\), let
\[
M_x(\eta)=\Bigl\{y\sim x: \eta(y)=\min_{z\sim x}\eta(z)\Bigr\}
\]
be the set of neighbors of minimum height. If \(F(\eta)=\emptyset\), the chain stays fixed. Otherwise, one free particle is chosen uniformly among all particles at free vertices, and then it moves to a uniformly chosen vertex in \(M_x(\eta)\). Equivalently, if
\[
S(\eta)=\sum_{z\in F(\eta)}\eta(z),
\]
then
\[
P\bigl(\eta,\eta-\delta_x+\delta_y\bigr)
=\frac{\eta(x)}{S(\eta)}\frac1{|M_x(\eta)|},
\qquad x\in F(\eta),\ y\in M_x(\eta).
\]
All other transition probabilities are zero, except for the holding probability in states with \(F(\eta)=\emptyset\).

The chain is generally reducible. If the total number of particles is \(m\), then the configurations in which every height is either \(\lfloor m/|V|\rfloor\) or \(\lceil m/|V|\rceil\) form a closed communicating class. When \(m\) is a multiple of \(|V|\), this class consists of the single absorbing state in which every vertex has height \(m/|V|\).

In this note we take \(G=C_n\), the cycle with \(n\) vertices, and start with
\[
\eta_0=n\delta_o
\]
for some vertex \(o\). Thus there are \(n\) particles on \(n\) vertices, and the unique absorbing state has height one at every vertex. Let \(\tau_n\) be the hitting time of this state.

With this notation, the main result of the note is the cubic upper bound
\[
\E_o \tau_n \le C n^3
\]
for an absolute constant \(C<\infty\); this is Theorem~\ref{thm:upper}. We also prove the deterministic transport lower bound \(\E_o\tau_n=\Omega(n^2)\), stated as Theorem~\ref{thm:lower}. The simulations reported below suggest that the true order is cubic, and that the missing matching lower bound should come from the maximum-height-two phase.

The rest of the paper consists of the proof of the upper bound and a section containing the current lower bound and simulations. Sections~\ref{sec:geometry}--\ref{sec:upper} give the one-dimensional geometry of level sets, the moment estimates, and the potential-function drift argument. Section~\ref{sec:lower} proves the current quadratic lower bound and reports the Monte Carlo simulation. Section~\ref{sec:extensions} contains brief comments on possible extensions.

For \(H\ge 1\), define the level set
\[
L_H(t)=\{x\in C_n:\eta_t(x)\ge H\},
\qquad M_H(t)=|L_H(t)|,
\]
and let
\[
M_0(t)=|\{x\in C_n:\eta_t(x)=0\}|
\]
be the number of holes.

\section{Geometry of level sets}\label{sec:geometry}

We shall use the following elementary separation property of the process on a cycle.

\begin{lemma}[separation]
Before stabilization, let \(x\neq y\) be two vertices and put \(m=\min\{\eta(x),\eta(y)\}\). If \(m>1\), then the two vertices split the cycle into two arcs, one of which contains a hole, while the other has no vertex of height less than \(m-1\).
\end{lemma}

\begin{proof}
There is always at least one hole before stabilization, since the total number of particles equals the total number of vertices. Initially the statement is true. Suppose it first fails after one transition. If neither endpoint changes height, then a vertex in the high arc must have moved from height \(m-1\) to height \(m-2\), which is impossible because all its neighbors in that arc have height at least \(m-1\). Lowering one endpoint cannot create the first violation, because it only lowers the threshold \(m-1\). Thus the only remaining possibility is that one endpoint, say \(x\), has just increased from \(m-1\) to \(m\). The particle that entered \(x\) came from a neighbor of height at least \(m\), and this gives the same violation one step earlier, a contradiction.
\end{proof}

When two vertices of height at least \(H\ge2\) are fixed, we call the arc containing no hole and no vertex below height \(H-1\) the corresponding \emph{high arc}.

\begin{corollary}\label{cor:diameter}
For \(H\ge2\), the diameter of \(L_H(t)\), measured along the high arc, satisfies
\[
\diam(L_H(t))<\frac{2n}{H}.
\]
\end{corollary}

\begin{proof}
If two vertices of \(L_H\) are at distance \(d\) along the high arc, then that arc contains at least \(d\) vertices of height at least \(H-1\). Since the total mass is \(n\), we have \(d(H-1)<n\). For \(H\ge2\), this gives \(d<2n/H\).
\end{proof}

\section{Moment of inertia}\label{sec:moment}

For a finite multiset \(A=\{a_1,\ldots,a_m\}\subset\mathbb R\), let
\[
\mu(A)=\frac1m\sum_{i=1}^m a_i
\]
and define its moment of inertia by
\[
I(A)=\sum_{i=1}^m (a_i-\mu(A))^2
=\sum_{i=1}^m a_i^2-m\mu(A)^2.
\]
This quantity is invariant under translation of all elements of \(A\). We use the convention \(I(\emptyset)=0\).

To define the moment of inertia of a level set on the cycle, cut the cycle at a hole and identify the remaining vertices with \(0,1,\ldots,n-1\). By the separation lemma, for \(H\ge2\) the set \(L_H\) lies in a single high arc, so this definition is independent of the chosen hole up to translation. We write
\[
I_H(t)=I\bigl(\{i:v_i\in L_H(t)\}\bigr),\qquad H\ge2.
\]
Similarly, by cutting at a vertex of height at least two, define
\[
I_0(t)=I\bigl(\{i:\eta_t(v_i)=0\}\bigr),
\]
the moment of inertia of the holes. If there is no vertex of height at least two, the process has already stabilized, and \(I_0\) is irrelevant.

We need three elementary estimates.

\begin{lemma}[two-sided movement]\label{lem:two-sided}
Let \(A=\{a_1,\ldots,a_m\}\), with \(m>1\), and suppose \(a_p<a_q\). Let \(A^-\) be obtained by replacing \(a_p\) by \(a_p-1\), and let \(A^+\) be obtained by replacing \(a_q\) by \(a_q+1\). If \(k=a_q-a_p\), then
\[
\frac12 I(A^-)+\frac12 I(A^+)
= I(A)+k+1-\frac1m .
\]
\end{lemma}

\begin{proof}
By translation invariance, assume \(\mu(A)=0\). Then \(\mu(A^-)=-1/m\) and \(\mu(A^+)=1/m\). Expanding the definition of \(I\) gives
\[
\begin{aligned}
\frac12\{I(A^-)+I(A^+)\}
&=I(A)+\frac12\{-2a_p+1-1/m+2a_q+1-1/m\} \\
&=I(A)+k+1-1/m .
\end{aligned}
\]
\end{proof}

\begin{lemma}[moving one point]\label{lem:moving-one}
Let \(A\) be a finite multiset of diameter \(D\), and let \(A'\) be obtained from \(A\) by moving one point by distance one. Then
\[
|I(A')-I(A)|\le 2D+1.
\]
In particular, \(I(A)-I(A')\le 2D+1\).
\end{lemma}

\begin{proof}
Write \(m=|A|\), let \(a\in A\) be the point that moves, and let \(\mu=\mu(A)\). If \(a\) is replaced by \(a+\varepsilon\), where \(\varepsilon=\pm1\), then
\[
I(A')-I(A)=2\varepsilon(a-\mu)+1-\frac1m .
\]
Since \(|a-\mu|\le D\), the claim follows.
\end{proof}

\begin{lemma}[adding and removing points]\label{lem:add-remove}
Let \(A\) be a finite multiset of diameter \(D\).
\begin{enumerate}
\item[(i)] For any \(b\in\mathbb R\), \(I(A\cup\{b\})\ge I(A)\).
\item[(ii)] If \(b\in A\), then \(I(A)-I(A\setminus\{b\})\le D^2\).
\end{enumerate}
\end{lemma}

\begin{proof}
For (i), the case \(A=\emptyset\) is immediate from the convention \(I(\emptyset)=0\). Otherwise, if \(m=|A|\), then
\[
I(A\cup\{b\})=I(A)+\frac{m}{m+1}(b-\mu(A))^2\ge I(A).
\]
For (ii), if \(|A|=1\), then \(D=0\), and the convention \(I(\emptyset)=0\) makes the claim trivial. Otherwise put \(B=A\setminus\{b\}\). Then
\[
I(A)=I(B)+\frac{|B|}{|B|+1}(b-\mu(B))^2.
\]
Both \(b\) and \(\mu(B)\) lie in the convex hull of \(A\), so \(|b-\mu(B)|\le D\).
\end{proof}

\section{The potential and the cubic upper bound}\label{sec:upper}

Fix constants by setting \(c_n=0\) and, for \(2\le H\le n\),
\[
c_{H-1}-c_H=64\left(\frac nH\right)^2 .
\]
Define
\[
G(t)=\sum_{H=2}^n I_H(t)+\sum_{H=1}^n c_HM_H(t)+I_0(t)-128n^2M_0(t).
\]
The terms \(c_HM_H\) reward particles that have fallen to lower levels. The last term rewards the filling of holes, and is used only to control the height-one part of the dynamics.

\begin{lemma}[drift]\label{lem:drift}
There is an absolute constant \(\gamma>0\) such that, whenever the chain has not stabilized and \(M_0(t)>1\),
\[
\E\bigl[G(t+1)-G(t)\mid \eta_t\bigr]\ge \gamma .
\]
\end{lemma}

\begin{proof}
Fix a state \(\eta\) with \(M_0>1\), and put \(S=S(\eta)\). A transition atom is a pair \(e=(x,y)\), where \(x\in F(\eta)\) is the selected source and \(y\in M_x(\eta)\) is the chosen minimum-height neighbor. Its probability is
\[
\mathbb P(e)=\frac{\eta(x)}{S}\frac1{|M_x(\eta)|}.
\]
We write \(\Delta G(e)\) for the change of \(G\). The proof is an accounting argument for disjoint summands of this change. Set
\[
\Delta_{\mathrm{high}}(e)
=\sum_{H=2}^n \Delta I_H(e)+\sum_{H=1}^n c_H\Delta M_H(e),
\qquad
\Delta_0(e)=\Delta I_0(e)-128n^2\Delta M_0(e).
\]
Then
\[
\Delta G(e)=\Delta_{\mathrm{high}}(e)+\Delta_0(e).
\]
The high-level part pays for all changes of \(I_H\), \(H\ge2\), and of \(\sum c_HM_H\); the hole part is used only for changes of \(I_0\) and \(M_0\). Thus the same transition atom may appear in both discussions, but different summands of \(\Delta G(e)\) are being accounted for.

Put
\[
\alpha_H=c_{H-1}-c_H=64\left(\frac nH\right)^2,
\qquad H\ge2,
\]
and let \(D_H=\diam(L_H)\). By Corollary~\ref{cor:diameter}, \(D_H<2n/H\). Lemmas~\ref{lem:moving-one} and~\ref{lem:add-remove} imply
\[
\text{moving one boundary point of }L_H\text{ can decrease }I_H
\text{ by at most }2D_H+1\le 5\left(\frac nH\right)^2,
\tag{4.1}
\]
and
\[
\text{removing one point from }L_H\text{ can decrease }I_H
\text{ by at most }D_H^2\le 4\left(\frac nH\right)^2 .
\tag{4.2}
\]
Adding a point to a level set never decreases its moment of inertia. Consequently, whenever a fall produces the increment \(\alpha_H\), it has enough surplus to pay for any one possible level-\(H\) loss:
\[
\alpha_H-5\left(\frac nH\right)^2
\ge 59\left(\frac nH\right)^2 .
\tag{4.3}
\]

\emph{High levels.} Fix \(H\ge2\). Decompose \(L_H\) into its components in the high arc. Only endpoints of such components can create a negative contribution to \(I_H\). Moreover, an endpoint of height strictly larger than \(H\) cannot create a negative level-\(H\) contribution: after the move it still belongs to \(L_H\), and the target either joins \(L_H\) or stays below it. Thus every possible loss at level \(H\) comes from an endpoint whose height is exactly \(H\).

Let \(J\) be a component of \(L_H\). If the two endpoints of \(J\) both have height \(H\), we group the two endpoint atoms of \(J\). Their source weights are equal. When both moves are along an \((H-1)\)-plateau, Lemma~\ref{lem:two-sided} gives a positive average change of \(I_H\). When at least one endpoint falls to a neighbor of height at most \(H-2\), the mass term contributes \(\alpha_H\), and the possible loss of \(I_H\) is bounded by (4.1) or (4.2), hence is paid by (4.3). This gives a positive conditional contribution, except for the following neutral case: \(L_H\) has a single point, both its neighbors have height \(H-1\), and the move merely shifts this single level-\(H\) point one step. Then the level-\(H\) contribution is exactly zero.

Now suppose that one endpoint \(x\) of \(J\) has height exactly \(H\), while the other endpoint \(z\) has height \(Q>H\). We call \(z\) the level-\(H\) parent of \(x\). The child atom at \(x\) has probability \(H/S\), since the outside neighbor of \(x\) is the unique minimum-height neighbor. For the parent, the outside neighbor relative to \(J\) has height at most \(H-1\), while the inside neighbor has height at least \(H\). Hence the parent atom that moves outward from \(J\) is also forced by the minimum-neighbor rule, and its probability is \(Q/S\ge H/S\).

If the child falls below level \(H\), its own mass increment \(\alpha_H\) pays for the level-\(H\) loss. If the child only moves the endpoint of \(L_H\) by one step along an \((H-1)\)-plateau, then the parent atom pays for that possible loss. Indeed, if the outside neighbor of the parent has height at most \(H-2\), the parent fall produces \(\alpha_H\), and (4.3) applies. If that outside neighbor has height exactly \(H-1\), the parent move adds a point to \(L_H\) just beyond the opposite endpoint; paired with the possible child move this is the two-sided movement of Lemma~\ref{lem:two-sided}, with the old parent point retained as an additional point. Retaining an additional point can only improve the estimate by Lemma~\ref{lem:add-remove}(i).

These charges are disjoint. A negative level-\(H\) term is attached to one component \(J\) of \(L_H\). If it is paid by a parent fall, only the \(H\)-th increment \(\alpha_H\) in that fall is used. For a fixed parent atom and a fixed level \(H\), the nested structure of level sets on the high arc gives at most one child endpoint at height \(H\); otherwise two distinct children of the same height would be endpoints of the same level interval on the same side of the parent. Therefore no increment \(\alpha_H\) is charged twice. Summing over all \(H\ge2\), the accounted high-level contribution is non-negative, and every non-neutral high-level block has conditional contribution bounded below by a universal constant \(c_1>0\).

It remains in the high-level accounting to handle the neutral one-point shoulder atoms. Such an atom at level \(H\) can occur only when \(L_H\) is a single point with two shoulders of height \(H-1\). Since we are assuming \(M_0>1\), this cannot happen with \(H=2\): one height-two vertex and no higher vertex would force exactly one hole by mass balance. Thus \(H\ge3\). Let \(J_{H-1}\) be the component of \(L_{H-1}\) containing this point and its two shoulders. Because there is a hole outside this component, \(J_{H-1}\) has two boundary endpoints. The outward moves from these endpoints are non-neutral level-\((H-1)\) moves, their total source weight is at least \(2(H-1)\), and the neutral source has weight \(H\). Grouping the neutral atom with these two lower-level boundary atoms therefore loses only an absolute factor. Each lower-level boundary atom is used for at most one such neutral atom, namely the one-point component of the next level inside the same \(L_{H-1}\)-component. After reducing \(c_1\), the high-level accounting still gives non-negative contribution everywhere and a positive conditional contribution on each high-level group not handled by the hole estimate below.

\emph{The hole coordinate.} We next account for \(\Delta_0\). Work in the coordinate system used to define \(I_0\), and let \(M=M_0>1\). For a height-one source, \(M_0\) does not change and exactly one hole moves by one step. Hence Lemma~\ref{lem:moving-one}, with diameter at most \(n\), gives the crude bound
\[
|\Delta_0(e)|=|\Delta I_0(e)|\le 2n+1 .
\tag{4.4}
\]
For a source of height at least two which sends a particle into a hole, one hole is deleted. Lemma~\ref{lem:add-remove}(ii) gives a possible loss of at most \(n^2\) in \(I_0\), while the term \(-128n^2M_0\) increases by \(128n^2\). Thus every such atom has hole contribution at least
\[
127n^2 .
\tag{4.5}
\]
If at least one atom of this type exists, then its atom weight is at least one: the source height is at least two and there are at most two minimum-height neighbors. The total atom weight of all height-one sources is at most \(n\). Combining (4.4) and (4.5), the total unnormalized hole contribution of all height-one atoms together with one high-to-hole atom is at least
\[
127n^2-n(2n+1)>100n^2
\]
for \(n\ge2\). Dividing by the total atom weight of this group gives a positive conditional contribution, uniformly in \(n\).

It remains to consider the case in which no vertex of height at least two can move directly into a hole. Then every hole which has a non-hole neighbor has a height-one neighbor on that side. The height-one atoms are exactly the exclusion-type moves of the holes through height-one sites. Write the hole clusters, in increasing order in the chosen coordinate system, as
\[
[\ell_1,r_1],\ldots,[\ell_k,r_k].
\]
For a hole at position \(p\), let \(\Delta_-(p)\) and \(\Delta_+(p)\) denote the change in \(I_0\) when that hole moves to \(p-1\) and \(p+1\), respectively. If \(\mu\) is the mean of the \(M\) hole positions, then
\[
\Delta_-(p)=-2(p-\mu)+1-\frac1M,
\qquad
\Delta_+(p)= 2(p-\mu)+1-\frac1M .
\tag{4.6}
\]
For a cluster \([\ell_i,r_i]\), the two outward moves contribute
\[
\Delta_-(\ell_i)+\Delta_+(r_i)
=2(r_i-\ell_i)+2-\frac2M .
\tag{4.7}
\]
If the gap between two consecutive clusters has length one, the single height-one site in the gap uses a uniform tie-break, so the two adjacent inward moves are counted with weights \(1/2\) and \(1/2\). If the gap has length at least two, they are counted with weights one. Since the sum of the two inward contributions across a gap is negative, replacing the weights \(1/2\) by weights one can only decrease the total. Therefore the actual total unnormalized contribution of all height-one atoms is at least the sum of the cluster quantities in (4.7):
\[
\begin{aligned}
\text{total hole contribution}
&\ge \sum_{i=1}^k\left(2(r_i-\ell_i)+2-\frac2M\right) \\
&=2M-\frac{2k}{M}
\ge 2M-2 .
\end{aligned}
\tag{4.8}
\]
The total atom weight of these height-one moves is at most \(2k\le2M\). Since \(M>1\), (4.8) gives conditional average at least \((2M-2)/(2M)\ge1/2\). Thus the hole contribution also has a uniformly positive conditional drift in the no-high-to-hole case.

In terms of the decomposition above, the high-level accounting proves
\[
\sum_e \mathbb P(e)\Delta_{\mathrm{high}}(e)\ge 0,
\]
and gives a positive conditional contribution on each non-neutral high-level group after the neutral shoulder atoms have been grouped with their adjacent non-neutral lower-level moves. The hole-coordinate accounting proves that the \(\Delta_0\)-contribution of the height-one atoms, together with the \(\Delta_0\)-part of any high-to-hole atom, is bounded below by a positive absolute constant in conditional average. These statements are about the two disjoint summands \(\Delta_{\mathrm{high}}\) and \(\Delta_0\), so there is no double counting even when the same transition atom affects both. Since the non-trivial transition atoms are exhausted by the high-level groups and the hole-coordinate groups, taking the minimum of the constants in the two estimates yields an absolute \(\gamma>0\) such that
\[
\E[\Delta G\mid \eta]\ge\gamma .
\]
This proves the lemma.
\end{proof}

\begin{theorem}[cubic upper bound]\label{thm:upper}
There exists an absolute constant \(C<\infty\) such that
\[
\E_o \tau_n \le C n^3
\]
for the pile process on \(C_n\) started from \(n\delta_o\).
\end{theorem}

\begin{proof}
Let
\[
\sigma=\inf\{t:M_0(t)\le1\}.
\]
By the drift lemma, \(G(t\wedge\sigma)-\gamma(t\wedge\sigma)\) is a submartingale.

We first bound \(G\). By the diameter estimate,
\[
I_H(t)\le M_H(t)\diam(L_H(t))^2
\le \frac nH\left(\frac{2n}{H}\right)^2
=\frac{4n^3}{H^3},
\]
so \(\sum_{H=2}^n I_H(t)\le 16n^3\). Also
\[
c_1=64n^2\sum_{H=2}^n H^{-2}<64n^2,
\]
and \(\sum_{H=1}^n M_H(t)=n\), hence
\[
\sum_{H=1}^n c_HM_H(t)\le 64n^3.
\]
Finally, \(I_0(t)\le n^3\). Thus \(G(t)\le 81n^3\) for all \(t\). On the other hand, \(G(0)\ge -128n^3\).

Applying optional stopping to \(t\wedge\sigma\) and then letting \(t\to\infty\),
\[
\gamma\E\sigma \le 81n^3-G(0)\le 209n^3.
\]
Therefore \(\E\sigma=O(n^3)\).

If \(M_0(\sigma)=0\), the chain has stabilized. If \(M_0(\sigma)=1\), then there is exactly one hole and exactly one vertex of height two. The pair consisting of the unique hole and the unique height-two vertex evolves as a nearest-neighbor birth-death chain on the cycle, with absorption when they meet; its expected absorption time is \(O(n^2)\). Hence
\[
\E\tau_n\le \E\sigma+O(n^2)=O(n^3).
\]
\end{proof}

\section{The current lower bound and simulations}\label{sec:lower}

The elementary deterministic lower bound is quadratic.

\begin{theorem}[transport lower bound]\label{thm:lower}
For the pile process on \(C_n\) started from \(n\delta_o\),
\[
\tau_n\ge \sum_{v\in C_n} d(o,v)=\left\lfloor\frac{n^2}{4}\right\rfloor.
\]
Consequently, \(\E_o\tau_n=\Omega(n^2)\).
\end{theorem}

\begin{proof}
Each transition moves exactly one particle across one edge. In the absorbing state, one particle must be present at every vertex. Thus the total number of moves is at least the minimum total transportation distance from the initial pile at \(o\) to one particle at each vertex. This minimum is exactly \(\sum_v d(o,v)\), which equals \(n^2/4\) when \(n\) is even and \((n^2-1)/4\) when \(n\) is odd.
\end{proof}

Simulations suggest that the quadratic lower bound is not sharp. We ran direct Monte Carlo simulations of the transition kernel, using the exact particle clock and uniform tie-break described in Section 1. The table reports the sample mean of \(\tau_n\), normalized by \(n^2\) and by \(n^3\), together with the number \(K_{T_2}\) of holes when the maximum height first drops to two. The runs shown below use \(233{,}202\) independent samples in total and about \(3.05\cdot 10^9\) simulated particle moves.

\begin{center}
\small
\begin{tabular}{rrrrr}
\hline
\(n\) & samples & \(\widehat{\E}\tau_n/n^2\) & \(\widehat{\E}\tau_n/n^3\) & \(\widehat{\E}K_{T_2}/n\) \\
\hline
16 & 50000 & 0.621 & 0.03880 & 0.374 \\
24 & 50000 & 0.823 & 0.03431 & 0.367 \\
32 & 50000 & 1.018 & 0.03180 & 0.361 \\
48 & 30000 & 1.407 & 0.02931 & 0.352 \\
64 & 30000 & 1.784 & 0.02788 & 0.348 \\
96 & 15000 & 2.527 & 0.02632 & 0.344 \\
128 & 5000 & 3.261 & 0.02548 & 0.341 \\
192 & 2000 & 4.724 & 0.02460 & 0.339 \\
256 & 1000 & 6.148 & 0.02401 & 0.337 \\
384 & 100 & 9.133 & 0.02378 & 0.335 \\
512 & 50 & 11.774 & 0.02300 & 0.336 \\
768 & 20 & 17.593 & 0.02291 & 0.334 \\
1024 & 32 & 23.300 & 0.02275 & 0.334 \\
\hline
\end{tabular}
\end{center}

The normalization by \(n^2\) grows steadily with \(n\), while the normalization by \(n^3\) appears to settle near a positive constant. The same simulations show that \(K_{T_2}/n\) is close to \(1/3\) for the larger values of \(n\), and that roughly \(84\%\) of the total time is spent after \(T_2\). This supports the view that the maximum-height-two phase is the bottleneck: at that time the state typically looks like a height-one plateau containing many vertices of height two, with an equal number of holes outside the plateau. The height-two vertices then behave like an exclusion process on the plateau until they are pushed into holes; see Liggett~\cite{Liggett1985} for background on exclusion processes.

A matching lower bound would follow from two additional ingredients. Let
\[
T_2=\inf\{t:\max_x\eta_t(x)\le2\},
\qquad K_{T_2}=M_0(T_2)=M_2(T_2).
\]
A matching cubic lower bound would follow from the following two estimates. First, for some constants \(a,p>0\) independent of \(n\),
\[
\mathbb P\{K_{T_2}\ge an\}\ge p.
\]
Second, on the event \(K_{T_2}\ge an\), one would need the conditional height-two estimate
\[
\mathbb E\bigl[\tau_n-T_2\mid \mathcal F_{T_2}\bigr]
\ge cK_{T_2}^3 .
\]
Together these statements would give the conjectural matching bound \(\E\tau_n=\Omega(n^3)\). At present, the unconditional lower bound proved above remains \(\Omega(n^2)\).

\section{Comments on possible extensions}\label{sec:extensions}

The same potential method should be adaptable to other one-dimensional graphs and to other initial configurations with total mass comparable to the number of vertices. Higher-dimensional graphs and trees require new ideas: the separation lemma is special to the cycle, and the one-dimensional moment calculation used above has no immediate two-dimensional analogue.

\end{document}